\documentclass{amsart}
\usepackage{amsfonts}

\newtheorem{theorem}{Theorem}[section]

\newtheorem{corollary}[theorem]{Corollary}

\newtheorem{definition}[theorem]{Definition}
\newtheorem{proposition}[theorem]{Proposition}
\newtheorem{remark}[theorem]{Remark}
\numberwithin{theorem}{section}
\input{tcilatex}

\begin{document}
\title[Radon-Nikodym type theorem for $\alpha $-completely positive maps on
groups ]{A Radon-Nikodym type theorem for $\alpha $-completely positive maps
on groups }
\author{Maria Joi\c{t}a}
\address{Department of Mathematics, Faculty of Mathematics and Computer
Science, University of Bucharest, Str. Academiei nr. 14, Bucharest, Romania}
\email{ mjoita@fmi.unibuc.ro}
\urladdr{http://sites.google.com/a/g.unibuc.ro/maria-joita/}

\begin{abstract}
We show that an operator valued $\alpha $-completely positive map on a group 
$G$ is given by a unitary representation of $G$ on a Krein space which
satisfies some condition. Moreover, two unitary equivalent such unitary
representations define the same $\alpha $-completely positive map. Also we
introduce a pre-order relation on the collection of $\alpha $-completely
positive maps on a group and we characterize this relation in terms of the
unitary representation associated to each map.
\end{abstract}

\maketitle

\section{Introduction}

The study of completely positive maps is motivated by their applications in
the theory of quantum measurements, operational approach to quantum
mechanics, quantum information theory, where operator-valued completely
positive maps on $C^{\ast }$-algebras are used as a mathematical model for
quantum operations, and quantum probability \cite{R, P, JS}. On the other
hand, the notion of locality in the Wightman formulation of gauge quantum
field theory conflicts with the notion of positivity. To avoid this,
Jakobczyk and Strocchi \cite{JS} introduced the concept of $\alpha $%
-positivity. Motivated by the notions of $\alpha $-positivity and
P-functional \cite{Ho, AO}, recently, Heo, Hong and Ji \cite{HHJ} introduced
the notion of $\alpha $-completely positive map between $C^{\ast }$%
-algebras, and they provided a Kasparov-Stinespring-Gelfand-Naimark-Segal
type construction for $\alpha $-completely positive maps.

In \cite{H}, Heo introduced the notion of $\alpha $-completely positive map
from a group $G$ to a $C^{\ast }$-algebra $A$. By analogy with the KSGNS\
construction for $\alpha $-completely positive maps on $C^{\ast }$-algebras 
\cite{HHJ}, he associated to an $\alpha $-completely positive map $\varphi $
from a group $G$ to the $C^{\ast }$-algebra $L(X)\ $of all adjointable
operators on a Hilbert $C^{\ast }$-module $X$ a quadruple $\left( \pi
_{\varphi },X_{\varphi },J_{\varphi },V_{\varphi }\right) $ consisting of a
Krein Hilbert $C^{\ast }$-module $\left( X_{\varphi },J_{\varphi }\right) $,
a $J_{\varphi }$-unitary representation of $G$ on $X_{\varphi }$ and a
bounded linear operator $V_{\varphi }$ such that $\left[ \pi _{\varphi
}\left( G\right) V_{\varphi }X\right] =X_{\varphi },V_{\varphi }^{\ast }\pi
_{\varphi }\left( g\right) ^{\ast }\pi _{\varphi }\left( g^{\prime }\right)
V_{\varphi }=V_{\varphi }^{\ast }\pi _{\varphi }\left( \alpha \left(
g^{-1}\right) g^{\prime }\right) V_{\varphi }$ for all $g,g^{\prime }\in G$
and $\varphi \left( g\right) =V_{\varphi }^{\ast }\pi _{\varphi }\left(
g\right) V_{\varphi }$ for all $g\in G$. But, in general, a such of
quadruple does not define an $\alpha $-completely positive map (Remark 2.6).
In this paper, we consider $\alpha $-completely positive maps from a group $%
G $ to $L(\mathcal{H})$, the $C^{\ast }$-algebra of all bounded linear
operators on a Hilbert space $\mathcal{H}$, and we show that under some
conditions, a quadruple $\left( \pi ,\mathcal{H},\mathcal{J},V\right) $
consisting of a Krein Hilbert space $\left( \mathcal{H},\mathcal{J}\right) $%
, a $\mathcal{J}$-unitary representation of $G$ on $\mathcal{H}$ and a
bounded linear operator $V$ defines an $\alpha $-completely positive map
andwe associate to each $\alpha $-completely positive map a such of
quadruple, that is unique up to unitary equivalence. In Section 3, we prove
a Radon-Nikodym theorem type for $\alpha $-completely positive maps on
groups.

\section{Stinespring type theorem for $\protect\alpha $-completely positive
maps }

Let $G$ be a (topological) group with an involution $\alpha $ (that is, a$\ $%
(continuous) map $\alpha :G\rightarrow G$ such that $\alpha ^{2}=$id$%
_{G},\alpha \left( e\right) =e$ and $\alpha \left( g^{-1}\right) =\alpha
\left( g\right) ^{-1}$ for all $g\in G$) and $\mathcal{H}$ a Hilbert space.

\begin{definition}
\cite[Definition 2.1]{H} A map $\varphi :G\rightarrow L(\mathcal{H})$ is $%
\alpha $-completely positive if

\begin{enumerate}
\item $\varphi \left( \alpha \left( g_{1}\right) \alpha \left( g_{2}\right)
\right) =\varphi \left( \alpha \left( g_{1}g_{2}\right) \right) =\varphi
\left( g_{1}g_{2}\right) $ for all $g_{1},g_{2}\in G;$

\item for all $g_{1},...,g_{n}\in G$, the matrix $\left[ \varphi \left(
\alpha \left( g_{i}\right) ^{-1}g_{j}\right) \right] _{i,j=1}^{n}$ is
positive in $L(\mathcal{H})$;

\item there is $K>0$, such that 
\begin{equation*}
\left[ \varphi \left( g_{i}\right) ^{\ast }\varphi \left( g_{j}\right) %
\right] _{i,j=1}^{n}\leq K\left[ \varphi \left( \alpha \left( g_{i}\right)
^{-1}g_{j}\right) \right] _{i,j=1}^{n}
\end{equation*}%
for all $g_{1},...,g_{n}\in G;$

\item for all $g\in G$, there is $M(g)>0$ such that 
\begin{equation*}
\left[ \varphi \left( \alpha \left( gg_{i}\right) ^{-1}gg_{j}\right) \right]
_{i,j=1}^{n}\leq M(g)\left[ \varphi \left( \alpha \left( g_{i}\right)
^{-1}g_{j}\right) \right] _{i,j=1}^{n}
\end{equation*}%
for all $g_{1},...,g_{n}\in G$.
\end{enumerate}
\end{definition}

\begin{remark}
Let $\varphi :G\rightarrow L(\mathcal{H})$ be an $\alpha $-completely
positive map.

\begin{enumerate}
\item $\varphi \left( \alpha \left( g\right) \right) =\varphi \left(
g\right) $ for all $g\in G;$

\item $\varphi \left( \alpha \left( g^{-1}\right) \right) =\varphi \left(
g\right) ^{\ast }$ for all $g\in G;$

\item $\varphi \left( g^{-1}\right) =\varphi \left( g\right) ^{\ast }$ for
all $g\in G.$
\end{enumerate}
\end{remark}

Let $\mathcal{H}$ be a Hilbert space and $\mathcal{J}$ a bounded linear
operator on $\mathcal{H}$ such that $\mathcal{J=J}^{\ast }\mathcal{=J}^{-1}.$%
Then we can define an indefinite inner product by $[x,y]=\left\langle 
\mathcal{J}x,y\right\rangle $. The pair $\left( \mathcal{H},\mathcal{J}%
\right) $ is called a Krein space. A representation of $G$ on the Krein
space $\left( \mathcal{H},\mathcal{J}\right) $ is a morphism $\pi
:G\rightarrow L(\mathcal{H})$. A $\mathcal{J}$-unitary representation of $G$
on the Krein space $\left( \mathcal{H},\mathcal{J}\right) $ is a
representation $\pi $ such that $\pi \left( g^{-1}\right) =\mathcal{J}\pi
\left( g\right) ^{\ast }\mathcal{J}$ for all $g\in G$ and $\pi \left(
e\right) =$id$_{\mathcal{H}}$. If $\pi $ is a representation of $G$ on $%
\mathcal{H}$, $\left[ \pi \left( G\right) \mathcal{H}\right] $ denotes the
closed linear subspace of $\mathcal{H}$ generated by $\{\pi \left( g\right)
\xi ;g\in G,\xi \in \mathcal{H}\}.$

\begin{theorem}
\cite[Theorem 2.2]{H} Let $\varphi :G\rightarrow L(\mathcal{H})$ be an $%
\alpha $-completely positive map. Then there are a Krein space $\left( 
\mathcal{H}_{\varphi },\mathcal{J}_{\varphi }\right) $, a $\mathcal{J}%
_{\varphi }$-unitary representation $\pi _{\varphi }$ of $G$ on $\left( 
\mathcal{H}_{\varphi },\mathcal{J}_{\varphi }\right) $ and a bounded linear
operator $V_{\varphi }:\mathcal{H}\rightarrow \mathcal{H}_{\varphi }$ such
that

\begin{enumerate}
\item $\varphi \left( g\right) =V_{\varphi }^{\ast }\pi _{\varphi }\left(
g\right) V_{\varphi }$ for all $g\in G;$

\item $\left[ \pi _{\varphi }\left( G\right) V_{\varphi }\mathcal{H}\right] =%
\mathcal{H}_{\varphi }$ $;$

\item $V_{\varphi }^{\ast }\pi _{\varphi }\left( g\right) ^{\ast }\pi
_{\varphi }\left( g^{\prime }\right) V_{\varphi }=V_{\varphi }^{\ast }\pi
_{\varphi }\left( \alpha \left( g^{-1}\right) g^{\prime }\right) V_{\varphi
} $ for all $g,g^{\prime }\in G$.
\end{enumerate}
\end{theorem}

The quadruple $\left( \pi _{\varphi },\mathcal{H}_{\varphi },\mathcal{J}%
_{\varphi },V_{\varphi }\right) $ is called the minimal Naimark -KSGNS
dilation of $\varphi $ \cite{H}.

\begin{remark}
If $\left( \pi _{\varphi },\mathcal{H}_{\varphi },\mathcal{J}_{\varphi
},V_{\varphi }\right) $ is the minimal Naimark -KSGNS dilation of $\varphi $
in the sense of Heo, then $\left( \pi _{\varphi },\mathcal{H}_{\varphi },%
\mathcal{J}_{\varphi },W\right) $, where $W=\mathcal{J}_{\varphi }$ $%
V_{\varphi }$, is a minimal Naimark -KSGNS dilation of $\varphi $ too.
Indeed, we have:

\begin{enumerate}
\item $\varphi \left( g\right) =$ $\varphi \left( g^{-1}\right) ^{\ast
}=\left( V_{\varphi }^{\ast }\pi _{\varphi }\left( g^{-1}\right) V_{\varphi
}\right) ^{\ast }=\left( V_{\varphi }^{\ast }\mathcal{J}_{\varphi }\pi
_{\varphi }\left( g\right) ^{\ast }\mathcal{J}_{\varphi }V_{\varphi }\right)
^{\ast }$

$=\left( W^{\ast }\pi _{\varphi }\left( g\right) ^{\ast }W\right) ^{\ast
}=W^{\ast }\pi _{\varphi }\left( g\right) W\ $for all $g\in G.$

\item Since $V_{\varphi }^{\ast }\pi _{\varphi }\left( g\right) ^{\ast }\pi
_{\varphi }\left( g^{\prime }\right) V_{\varphi }=V_{\varphi }^{\ast }\pi
_{\varphi }\left( \alpha \left( g^{-1}\right) g^{\prime }\right) V_{\varphi
} $ for all $g,g^{\prime }\in G,$ and $\left[ \pi _{\varphi }\left( G\right)
V_{\varphi }\mathcal{H}\right] =\mathcal{H}_{\varphi }$, we have 
\begin{equation*}
V_{\varphi }^{\ast }\pi _{\varphi }\left( g\right) ^{\ast }=V_{\varphi
}^{\ast }\pi _{\varphi }\left( \alpha \left( g^{-1}\right) \right)
\end{equation*}%
for all $g\in G$, and then 
\begin{equation*}
\pi _{\varphi }\left( g\right) V_{\varphi }=\mathcal{J}_{\varphi }\pi
_{\varphi }\left( \alpha \left( g\right) \right) \mathcal{J}_{\varphi
}V_{\varphi }
\end{equation*}%
for all $g\in G.$ Then $\left[ \pi _{\varphi }\left( G\right) W\mathcal{H}%
\right] =\mathcal{J}_{\varphi }\left[ \mathcal{J}_{\varphi }\pi _{\varphi
}\left( \alpha \left( G\right) \right) V_{\varphi }\mathcal{H}\right] =%
\mathcal{J}_{\varphi }\left[ \pi _{\varphi }\left( G\right) V_{\varphi }%
\mathcal{H}\right] =\mathcal{J}_{\varphi }\mathcal{H}_{\varphi }=\mathcal{H}%
_{\varphi }.$

\item Let $g,g^{\prime }\in G$. Then 
\begin{eqnarray*}
W^{\ast }\pi _{\varphi }\left( g\right) ^{\ast }\pi _{\varphi }\left(
g^{\prime }\right) W &=&V_{\varphi }^{\ast }\mathcal{J}_{\varphi }\pi
_{\varphi }\left( g\right) ^{\ast }\pi _{\varphi }\left( g^{\prime }\right) 
\mathcal{J}_{\varphi }V_{\varphi }=V_{\varphi }^{\ast }\pi _{\varphi }\left(
g^{-1}\right) \mathcal{J}_{\varphi }\pi _{\varphi }\left( g^{\prime }\right) 
\mathcal{J}_{\varphi }V_{\varphi } \\
&=&V_{\varphi }^{\ast }\pi _{\varphi }\left( g^{-1}\right) \pi _{\varphi
}\left( \alpha \left( g^{\prime }\right) \right) V_{\varphi }\ \  \\
&=&V_{\varphi }^{\ast }\pi _{\varphi }\left( g^{-1}\alpha \left( g^{\prime
}\right) \right) V_{\varphi }=\varphi \left( g^{-1}\alpha \left( g^{\prime
}\right) \right) =\varphi \left( \alpha \left( g^{-1}\right) g^{\prime
}\right) \\
&=&\varphi \left( g^{\prime -1}\alpha \left( g\right) \right) ^{\ast
}=\left( V_{\varphi }^{\ast }\pi _{\varphi }\left( g^{\prime -1}\alpha
\left( g\right) \right) V_{\varphi }\right) ^{\ast } \\
&=&\left( V_{\varphi }^{\ast }\mathcal{J}_{\varphi }\pi _{\varphi }\left(
\alpha \left( g^{-1}\right) g^{\prime }\right) ^{\ast }\mathcal{J}_{\varphi
}V_{\varphi }\right) ^{\ast } \\
&=&\left( W^{\ast }\pi _{\varphi }\left( \alpha \left( g^{-1}\right)
g^{\prime }\right) ^{\ast }W\right) ^{\ast }=W^{\ast }\pi _{\varphi }\left(
\alpha \left( g^{-1}\right) g^{\prime }\right) W.
\end{eqnarray*}
\end{enumerate}
\end{remark}

\begin{remark}
We remark that $\mathcal{J}_{\varphi }\pi _{\varphi }\left( g\right)
V_{\varphi }=\pi _{\varphi }\left( \alpha \left( g\right) \right) V_{\varphi
}$ for all $g\in G$, if and only if, $V_{\varphi }^{\ast }\pi _{\varphi
}\left( g\right) ^{\ast }\pi _{\varphi }\left( g^{\prime }\right) V_{\varphi
}=V_{\varphi }^{\ast }\pi _{\varphi }\left( \alpha \left( g^{-1}\right)
g^{\prime }\right) V_{\varphi }$ for all $g,g^{\prime }\in G$ and $\mathcal{J%
}_{\varphi }V_{\varphi }=V_{\varphi }.$

Indeed, if $\mathcal{J}_{\varphi }\pi _{\varphi }\left( g\right) V_{\varphi
}=\pi _{\varphi }\left( \alpha \left( g\right) \right) V_{\varphi }$ for all 
$g\in G$, then $\mathcal{J}_{\varphi }V_{\varphi }=V_{\varphi }\ $and 
\begin{eqnarray*}
V_{\varphi }^{\ast }\pi _{\varphi }\left( g\right) ^{\ast }\pi _{\varphi
}\left( g^{\prime }\right) V_{\varphi } &=&V_{\varphi }^{\ast }\mathcal{J}%
_{\varphi }\pi _{\varphi }\left( g^{-1}\right) \mathcal{J}_{\varphi }\pi
_{\varphi }\left( g^{\prime }\right) V_{\varphi }\ \text{(Remark 2.4 (2))} \\
&=&V_{\varphi }^{\ast }\pi _{\varphi }\left( g^{-1}\right) \pi _{\varphi
}\left( \alpha \left( g^{\prime }\right) \right) V_{\varphi }=V_{\varphi
}^{\ast }\pi _{\varphi }\left( g^{-1}\alpha \left( g^{\prime }\right)
\right) V_{\varphi } \\
&=&\varphi \left( g^{-1}\alpha \left( g^{\prime }\right) \right) =\varphi
\left( \alpha \left( g^{-1}\right) g^{\prime }\right) \\
&=&V_{\varphi }^{\ast }\pi _{\varphi }\left( \alpha \left( g^{-1}\right)
g^{\prime }\right) V_{\varphi }
\end{eqnarray*}%
for all $g,g^{\prime }\in G.$

Conversely, if $V_{\varphi }^{\ast }\pi _{\varphi }\left( g\right) ^{\ast
}\pi _{\varphi }\left( g^{\prime }\right) V_{\varphi }=V_{\varphi }^{\ast
}\pi _{\varphi }\left( \alpha \left( g^{-1}\right) g^{\prime }\right)
V_{\varphi }$ for all $g,g^{\prime }\in G$, then%
\begin{equation*}
\pi _{\varphi }\left( g\right) V_{\varphi }=\mathcal{J}_{\varphi }\pi
_{\varphi }\left( \alpha \left( g\right) \right) \mathcal{J}_{\varphi
}V_{\varphi }
\end{equation*}%
for all $g\in G$, and taking into account that $\mathcal{J}_{\varphi
}V_{\varphi }=V_{\varphi }$, we have 
\begin{equation*}
\mathcal{J}_{\varphi }\pi _{\varphi }\left( g\right) V_{\varphi }=\pi
_{\varphi }\left( \alpha \left( g\right) \right) \mathcal{J}_{\varphi
}V_{\varphi }=\pi _{\varphi }\left( \alpha \left( g\right) \right)
V_{\varphi }
\end{equation*}%
for all $g\in G$.
\end{remark}

\begin{remark}
If $G$ is a group with an involution $\alpha ,$ $\pi $ is a $\mathcal{J}$%
-unitary representation of $G$ on $\left( \mathcal{K},\mathcal{J}\right) $
and $V$ a bounded linear operator from a Hilbert space $\mathcal{H\ }$to $%
\mathcal{K}$ such that $\left[ \pi \left( G\right) V\mathcal{H}\right] =%
\mathcal{K}$ and $V^{\ast }\pi \left( g\right) ^{\ast }\pi \left( g^{\prime
}\right) V=V^{\ast }\pi \left( \alpha \left( g^{-1}\right) g^{\prime
}\right) V$ for all $g,g^{\prime }\in G$, then the map $\varphi
:G\rightarrow L(\mathcal{H})$ defined by $\varphi \left( g\right) =V^{\ast
}\pi \left( g\right) V$ is not in general an $\alpha $-completely positive
map.

\textit{Example.} Let $\mathbb{Z}$ be the additive group of integers numbers
and $\alpha \left( n\right) =-n$ an involution of $\mathbb{Z}$. The map $%
\mathcal{J}:\mathbb{C}^{2}\rightarrow \mathbb{C}^{2}$ defined by $\mathcal{J}%
\left( x,y\right) =\left( y,x\right) $ is a bounded linear operator such
that $\mathcal{J}=\mathcal{J}^{\ast }=\mathcal{J}^{-1}$, the map $\pi :%
\mathbb{Z\rightarrow }L\mathbb{(C}^{2}\mathbb{)}$ defined $\pi \left(
n\right) \left( x,y\right) =\left( e^{n}x,e^{-n}y\right) $ is a $\mathcal{J}$%
-unitary representation of $\mathbb{Z}$ on $\left( \mathbb{C}^{2},\mathcal{J}%
\right) ,$ and the map $V:$ $\mathbb{C}^{2}\rightarrow \mathbb{C}^{2}$
defined by $V\left( x,y\right) =\left( x-y,y\right) $ is a bounded linear
operator. It is easy to verify that $\left[ \pi \left( \mathbb{Z}\right) V%
\mathbb{C}^{2}\right] =\mathbb{C}^{2}$ and $V^{\ast }\pi \left( n\right)
^{\ast }\pi \left( m\right) V=V^{\ast }\pi \left( n+m\right) V=V^{\ast }\pi
\left( \alpha \left( -n\right) m\right) V$ for all $n,m\in \mathbb{Z},$ but $%
\varphi :\mathbb{Z\rightarrow }L\mathbb{(C}^{2}\mathbb{)}$ defined by $%
\varphi \left( n\right) =V^{\ast }\pi \left( n\right) V$ is not $\alpha $%
-completely positive, because $\varphi \left( n\right) \neq \varphi \left(
-n\right) =\varphi \left( \alpha \left( n\right) \right) .$
\end{remark}

\begin{proposition}
Let $G$ be a group with an involution $\alpha ,$ $\pi $ a $\mathcal{J}$%
-unitary representation of $G$ on $\left( \mathcal{K},\mathcal{J}\right) $
and $V$ a bounded linear operator from a Hilbert space $\mathcal{H}$ such
that $\left[ \pi \left( G\right) V\mathcal{H}\right] =\mathcal{K}$ and $%
\mathcal{J}\pi \left( g\right) V=\pi \left( \alpha \left( g\right) \right) V$
for all $g\in G$. Then the map $\varphi :G\rightarrow L(\mathcal{H})$
defined by $\varphi \left( g\right) =V^{\ast }\pi \left( g\right) V$ is an $%
\alpha $-completely positive map.
\end{proposition}

\begin{proof}
It is similar to the proof of Proposition 3.1.
\end{proof}

\begin{theorem}
Let $\varphi :G\rightarrow L(\mathcal{H})$ be an $\alpha $-completely
positive map.

\begin{enumerate}
\item There are a Krein space $\left( \mathcal{H}_{\varphi },\mathcal{J}%
_{\varphi }\right) $, a $\mathcal{J}_{\varphi }$-unitary representation $\pi
_{\varphi }$ of $G$ on $\left( \mathcal{H}_{\varphi },\mathcal{J}_{\varphi
}\right) $ and a bounded linear operator $V_{\varphi }:\mathcal{H}%
\rightarrow \mathcal{H}_{\varphi }$ such that

\begin{enumerate}
\item $\varphi \left( g\right) =V_{\varphi }^{\ast }\pi _{\varphi }\left(
g\right) V_{\varphi }$ for all $g\in G;$

\item $\left[ \pi _{\varphi }\left( G\right) V_{\varphi }\mathcal{H}\right] =%
\mathcal{H}_{\varphi }$ $;$

\item $\mathcal{J}_{\varphi }\pi _{\varphi }\left( g\right) V_{\varphi }=\pi
_{\varphi }\left( \alpha \left( g\right) \right) V_{\varphi }$ for all $g\in
G$.
\end{enumerate}

\item If $\pi $ is a $\mathcal{J}$-unitary representation of $G$ on a Krein
space $\left( \mathcal{K},\mathcal{J}\right) \ $and $V:$ $\mathcal{H}%
\rightarrow \mathcal{K}$ is a bounded linear operator such that

\begin{enumerate}
\item $\varphi \left( g\right) =V^{\ast }\pi \left( g\right) V$ for all $%
g\in G;$

\item $\left[ \pi \left( G\right) V\mathcal{H}\right] =\mathcal{K}$ $;$

\item $\mathcal{J}\pi \left( g\right) V=\pi \left( \alpha \left( g\right)
\right) V$ for all $g\in G$.

then there is a unitary operator $U:\mathcal{H}_{\varphi }\rightarrow 
\mathcal{K}$ such that

\begin{enumerate}
\item $U\mathcal{J}_{\varphi }=\mathcal{J}U;$

\item $UV_{\varphi }=V;$

\item $U\pi _{\varphi }(g)=\pi \left( g\right) U$ for all $g\in G.$
\end{enumerate}
\end{enumerate}
\end{enumerate}
\end{theorem}

\begin{proof}
(1). We will give a sketch of proof (see \cite[Theorem 2.2]{H} and Remark
2.5 for the detailed proof). Let $\mathcal{F}(G,\mathcal{H})$ be the vector
space of all functions from $G$ to $\mathcal{H}$ with finite support. The
map $\left\langle \cdot ,\cdot \right\rangle :\mathcal{F}(G,\mathcal{H}%
)\times \mathcal{F}(G,\mathcal{H})\rightarrow \mathbb{C}$ defined by 
\begin{equation*}
\left\langle f_{1},f_{2}\right\rangle =\tsum\limits_{g,g^{\prime
}}\left\langle f_{1}\left( g\right) ,\varphi \left( \alpha \left(
g^{-1}\right) g^{\prime }\right) f_{2}\left( g^{\prime }\right) \right\rangle
\end{equation*}%
is a positive semi-definite sesquilinear form and $\mathcal{H}_{\varphi }$
is the Hilbert space obtained by the completion of the pre-Hilbert space $%
\mathcal{F}(G,\mathcal{H})/\mathcal{N}_{\varphi }$, where $\mathcal{N}%
_{\varphi }=\{f\in \mathcal{F}(G,\mathcal{H})/$\ $\left\langle
f,f\right\rangle =0\}$.

The linear map $\mathcal{J}_{\varphi }:\mathcal{F}(G,\mathcal{H})\rightarrow 
\mathcal{F}(G,\mathcal{H})$ given by $\mathcal{J}_{\varphi }\left( f\right)
=f\circ \alpha $ extends to a bounded linear operator $\mathcal{J}_{\varphi
}:\mathcal{H}_{\varphi }\rightarrow \mathcal{H}_{\varphi }.$ Moreover, $%
\mathcal{J}_{\varphi }=\mathcal{J}_{\varphi }^{\ast }=\mathcal{J}_{\varphi
}^{-1}\ $and $\left( \mathcal{H}_{\varphi },\mathcal{J}_{\varphi }\right) $
is a Krein space. For each $g\in G$, the map $\pi _{\varphi }\left( g\right)
:\mathcal{F}(G,\mathcal{H})\rightarrow \mathcal{F}(G,\mathcal{H})$ given by $%
\pi _{\varphi }\left( g\right) \left( f\right) \left( g^{\prime }\right)
=f\left( g^{-1}g^{\prime }\right) $ extends to a bounded linear operator
from $\mathcal{H}_{\varphi }$ to $\mathcal{H}_{\varphi }$, and the map $%
g\mapsto \pi _{\varphi }\left( g\right) $ is a $\mathcal{J}_{\varphi }$%
-unitary representation $\pi _{\varphi }$ of $G$ on $\left( \mathcal{H}%
_{\varphi },\mathcal{J}_{\varphi }\right) $. The linear map $V_{\varphi }:%
\mathcal{H}\rightarrow \mathcal{F}(G,\mathcal{H})$ given by $V_{\varphi }\xi
=\xi \delta _{e}$, where $\delta _{e}:G\rightarrow \mathbb{C},\delta
_{e}\left( g\right) =0$ if $g\neq e$ and $\delta _{e}\left( e\right) =1.$

(2). We consider the linear map $U:$span$\{\pi _{\varphi }\left( g\right)
V_{\varphi }\xi /g\in G,\xi \in \mathcal{H}\}\rightarrow $span$\{\pi \left(
g\right) V\xi /g\in G,\xi \in \mathcal{H}\}$ defined by 
\begin{equation*}
U\left( \pi _{\varphi }\left( g\right) V_{\varphi }\xi \right) =\pi \left(
g\right) V\xi .
\end{equation*}%
Since 
\begin{eqnarray*}
&&\left\langle U\left( \tsum_{i=1}^{n}\pi _{\varphi }\left( g_{i}\right)
V_{\varphi }\xi _{i}\right) ,U\left( \tsum_{j=1}^{m}\pi _{\varphi }\left(
g_{j}^{\prime }\right) V_{\varphi }\zeta _{j}\right) \right\rangle \\
&=&\tsum_{i=1}^{n}\tsum_{j=1}^{m}\left\langle \pi \left( g_{i}\right) V\xi
_{i},\pi \left( g_{j}^{\prime }\right) V\zeta _{j}\right\rangle \\
&=&\tsum_{i=1}^{n}\tsum_{j=1}^{m}\left\langle \left( V^{\ast }\pi \left(
g_{j}^{\prime }\right) ^{\ast }\pi \left( g_{i}\right) V\xi _{i}\right)
,\zeta _{j}\right\rangle \\
&=&\tsum_{i=1}^{n}\tsum_{j=1}^{m}\left\langle V^{\ast }\pi \left( \alpha
\left( g_{j}^{\prime -1}\right) g_{i}\right) V\xi _{i},\zeta
_{j}\right\rangle \\
&=&\tsum_{i=1}^{n}\tsum_{j=1}^{m}\left\langle \varphi \left( \alpha \left(
g_{j}^{\prime -1}\right) g_{i}\right) \xi _{i},\zeta _{j}\right\rangle \\
&=&\tsum_{i=1}^{n}\tsum_{j=1}^{m}\left\langle V_{\varphi }^{\ast }\pi
_{\varphi }\left( \alpha \left( g_{j}^{\prime -1}\right) g_{i}\right)
V_{\varphi }\xi _{i},\zeta _{j}\right\rangle \\
&=&\left\langle \tsum_{i=1}^{n}\pi _{\varphi }\left( g_{i}\right) V_{\varphi
}\xi _{i},\tsum_{j=1}^{m}\pi _{\varphi }\left( g_{j}^{\prime }\right)
V_{\varphi }\zeta _{j}\right\rangle
\end{eqnarray*}%
for all $g_{1},...,g_{n},g_{1}^{\prime },...,g_{m}^{\prime }\in G$ and for
all $\xi _{1},...,\xi _{n},\zeta _{1},...,\zeta _{m}\in \mathcal{H},$ $U$
extends to a unitary operator $U$ from $\mathcal{H}_{\varphi }$ to $\mathcal{%
K}$. Moreover, $U\pi _{\varphi }(g)=\pi \left( g\right) U$ for all $g\in G$
and $UV_{\varphi }=V$. Since 
\begin{eqnarray*}
U\mathcal{J}_{\varphi }\left( \pi _{\varphi }\left( g\right) V_{\varphi }\xi
\right) &=&U\left( \pi _{\varphi }\left( \alpha \left( g\right) \right)
V_{\varphi }\xi \right) =\pi \left( \alpha \left( g\right) \right) V\xi \\
&=&\mathcal{J}\left( \pi \left( g\right) V\xi \right) =\mathcal{J}U\left(
\pi _{\varphi }\left( g\right) V_{\varphi }\xi \right)
\end{eqnarray*}%
for all $g\in G$ and for all $\xi \in \mathcal{H}$, and since $\left[ \pi
_{\varphi }\left( G\right) V_{\varphi }\mathcal{H}\right] =\mathcal{H}%
_{\varphi },$ we have $U\mathcal{J}_{\varphi }=\mathcal{J}U.$
\end{proof}

If $G$ is a topological group and $\varphi $ is bounded, then the $\mathcal{J%
}_{\varphi }$-unitary representation $\pi _{\varphi }$ is strictly
continuous.

The triple $\left( \pi _{\varphi },\left( \mathcal{H}_{\varphi },\mathcal{J}%
_{\varphi }\right) ,V_{\varphi }\right) $ is called the minimal Stinespring
construction associated to $\varphi $.

\section{ Radon-Nicodym type theorem for $\protect\alpha $-completely
positive maps}

Let $G$ be a group with an involution $\alpha $, $\mathcal{H}$ a Hilbert
space and $\alpha -CP(G,\mathcal{H})=\{\varphi :G\rightarrow L(\mathcal{H}%
);\varphi $ is $\alpha $-completely positive$\}$.

Let $\varphi \in \alpha -CP(G,\mathcal{H})$ and let $\left( \pi _{\varphi
},\left( \mathcal{H}_{\varphi },\mathcal{J}_{\varphi }\right) ,V_{\varphi
}\right) $ be the minimal Stinespring construction associated to $\varphi .$

\begin{proposition}
Let $T\in $ $\pi _{\varphi }\left( G\right) ^{\prime }\subseteq L\left( 
\mathcal{H}_{\varphi }\right) $ such that $\ T\geq 0\ $and $T\mathcal{J}%
_{\varphi }=\mathcal{J}_{\varphi }T$. Then the map $\varphi
_{T}:G\rightarrow L(\mathcal{H}$ $)$ defined by $\varphi _{T}\left( g\right)
=V_{\varphi }^{\ast }T\pi _{\varphi }\left( g\right) V_{\varphi }$ is $%
\alpha $-completely positive.
\end{proposition}

\begin{proof}
From 
\begin{eqnarray*}
\varphi _{T}\left( \alpha \left( g_{1}\right) \alpha \left( g_{2}\right)
\right) &=&V_{\varphi }^{\ast }T\pi _{\varphi }\left( \alpha \left(
g_{1}\right) \right) \pi _{\varphi }\left( \alpha \left( g_{2}\right)
\right) V_{\varphi } \\
&=&V_{\varphi }^{\ast }T\pi _{\varphi }\left( \alpha \left( g_{1}\right)
\right) \mathcal{J}_{\varphi }\pi _{\varphi }\left( g_{2}\right) V_{\varphi }
\\
&=&V_{\varphi }^{\ast }\mathcal{J}_{\varphi }\pi _{\varphi }\left( \alpha
\left( g_{1}\right) \right) \mathcal{J}_{\varphi }T\pi _{\varphi }\left(
g_{2}\right) V_{\varphi } \\
&=&V_{\varphi }^{\ast }\pi _{\varphi }\left( \alpha \left( g_{1}^{-1}\right)
\right) ^{\ast }T\pi _{\varphi }\left( g_{2}\right) V_{\varphi } \\
&=&V_{\varphi }^{\ast }\pi _{\varphi }\left( g_{1}^{-1}\right) ^{\ast }%
\mathcal{J}_{\varphi }T\pi _{\varphi }\left( g_{2}\right) V_{\varphi } \\
&=&V_{\varphi }^{\ast }\mathcal{J}_{\varphi }\pi _{\varphi }\left(
g_{1}\right) T\pi _{\varphi }\left( g_{2}\right) V_{\varphi } \\
&=&V_{\varphi }^{\ast }T\pi _{\varphi }\left( g_{1}g_{2}\right) V_{\varphi
}=\varphi _{T}\left( g_{1}g_{2}\right)
\end{eqnarray*}%
and 
\begin{eqnarray*}
\varphi _{T}\left( \alpha \left( g_{1}g_{2}\right) \right) &=&V_{\varphi
}^{\ast }T\pi _{\varphi }\left( \alpha \left( g_{1}g_{2}\right) \right)
V_{\varphi }=V_{\varphi }^{\ast }T\mathcal{J}_{\varphi }\pi _{\varphi
}\left( g_{1}g_{2}\right) V_{\varphi } \\
&=&V_{\varphi }^{\ast }\mathcal{J}_{\varphi }T\pi _{\varphi }\left(
g_{1}g_{2}\right) V_{\varphi }=V_{\varphi }^{\ast }T\pi _{\varphi }\left(
g_{1}g_{2}\right) V_{\varphi }=\varphi _{T}\left( g_{1}g_{2}\right)
\end{eqnarray*}%
for all $g_{1},g_{2}\in G$, we deduce that $\varphi _{T}\left( \alpha \left(
g_{1}\right) \alpha \left( g_{2}\right) \right) =\varphi _{T}\left( \alpha
\left( g_{1}g_{2}\right) \right) =\varphi _{T}\left( g_{1}g_{2}\right) $ for
all $g_{1},g_{2}\in G$.

Let $g_{1},...,g_{n}\in G$ and $\xi _{1},...,\xi _{n}\in \mathcal{H}$. Then 
\begin{eqnarray*}
&&\left\langle \left[ \varphi _{T}\left( \alpha \left( g_{i}^{-1}\right)
g_{j}\right) \right] _{i,j=1}^{n}\left( \xi _{k}\right) _{k=1}^{n},\left(
\xi _{k}\right) _{k=1}^{n}\right\rangle \\
&=&\tsum\limits_{i,j=1}^{n}\left\langle \varphi _{T}\left( \alpha \left(
g_{i}^{-1}\right) g_{j}\right) \xi _{j},\xi _{i}\right\rangle \\
&=&\tsum\limits_{i,j=1}^{n}\left\langle V_{\varphi }^{\ast }T\pi _{\varphi
}\left( \alpha \left( g_{i}^{-1}\right) g_{j}\right) V_{\varphi }\xi
_{j},\xi _{i}\right\rangle \\
&=&\tsum\limits_{i,j=1}^{n}\left\langle T\pi _{\varphi }\left( g_{j}\right)
V_{\varphi }\xi _{j},\pi _{\varphi }\left( \alpha \left( g_{i}^{-1}\right)
\right) ^{\ast }V_{\varphi }\xi _{i}\right\rangle \\
&=&\tsum\limits_{i,j=1}^{n}\left\langle T\pi _{\varphi }\left( g_{j}\right)
V_{\varphi }\xi _{j},\mathcal{J}_{\varphi }\pi _{\varphi }\left( \alpha
\left( g_{i}\right) \right) \mathcal{J}_{\varphi }V_{\varphi }\xi
_{i}\right\rangle \\
&=&\tsum\limits_{i,j=1}^{n}\left\langle T\pi _{\varphi }\left( g_{j}\right)
V_{\varphi }\xi _{j},\mathcal{J}_{\varphi }\pi _{\varphi }\left( \alpha
\left( g_{i}\right) \right) V_{\varphi }\xi _{i}\right\rangle \\
&=&\left\langle T\tsum\limits_{j=1}^{n}\pi _{\varphi }\left( g_{j}\right)
V_{\varphi }\xi _{j},\tsum\limits_{i=1}^{n}\pi _{\varphi }\left(
g_{i}\right) V_{\varphi }\xi _{i}\right\rangle \geq 0
\end{eqnarray*}%
and

\begin{eqnarray*}
&&\left\langle \left[ \varphi _{T}\left( g_{i}\right) ^{\ast }\varphi
_{T}\left( g_{j}\right) \right] _{i,j=1}^{n}\left( \xi _{k}\right)
_{k=1}^{n},\left( \xi _{k}\right) _{k=1}^{n}\right\rangle \\
&=&\tsum\limits_{i,j=1}^{n}\left\langle V_{\varphi }^{\ast }T\pi _{\varphi
}\left( g_{j}\right) V_{\varphi }\xi _{j},V_{\varphi }^{\ast }T\pi _{\varphi
}\left( g_{i}\right) V_{\varphi }\xi _{i}\right\rangle \\
&=&\left\langle T^{\ast }V_{\varphi }V_{\varphi }^{\ast
}T\tsum\limits_{j=1}^{n}\pi _{\varphi }\left( g_{j}\right) V_{\varphi }\xi
_{j},\tsum\limits_{i=1}^{n}\pi _{\varphi }\left( g_{i}\right) V_{\varphi
}\xi _{i}\right\rangle \\
&\leq &\left\Vert V_{\varphi }\right\Vert ^{2}\left\Vert T\right\Vert
\tsum\limits_{i,j=1}^{n}\left\langle V_{\varphi }^{\ast }\mathcal{J}%
_{\varphi }\pi _{\varphi }\left( g_{i}^{-1}\right) \mathcal{J}_{\varphi
}T\pi _{\varphi }\left( g_{j}\right) V_{\varphi }\xi _{j},\xi
_{i}\right\rangle \\
&=&\left\Vert V_{\varphi }\right\Vert ^{2}\left\Vert T\right\Vert
\tsum\limits_{i,j=1}^{n}\left\langle V_{\varphi }^{\ast }T\pi _{\varphi
}\left( g_{i}^{-1}\right) \pi _{\varphi }\left( \alpha \left( g_{j}\right)
\right) V_{\varphi }\xi _{j},\xi _{i}\right\rangle \\
&=&\left\Vert V_{\varphi }\right\Vert ^{2}\left\Vert T\right\Vert
\tsum\limits_{i,j=1}^{n}\left\langle V_{\varphi }^{\ast }T\pi _{\varphi
}\left( g_{i}^{-1}\alpha \left( g_{j}\right) \right) V_{\varphi }\xi
_{j},\xi _{i}\right\rangle \\
&=&\left\Vert V_{\varphi }\right\Vert ^{2}\left\Vert T\right\Vert
\tsum\limits_{i,j=1}^{n}\left\langle \varphi _{T}\left( g_{i}^{-1}\alpha
\left( g_{j}\right) \right) \xi _{j},\xi _{i}\right\rangle \\
&=&\left\Vert V_{\varphi }\right\Vert ^{2}\left\Vert T\right\Vert
\tsum\limits_{i,j=1}^{n}\left\langle \varphi _{T}\left( \alpha \left(
g_{i}^{-1}\right) g_{j}\right) \xi _{j},\xi _{i}\right\rangle \\
&=&\left\Vert V_{\varphi }\right\Vert ^{2}\left\Vert T\right\Vert
\left\langle \left[ \varphi _{T}\left( \alpha \left( g_{i}^{-1}\right)
g_{j}\right) \right] _{i,j=1}^{n}\left( \xi _{k}\right) _{k=1}^{n},\left(
\xi _{k}\right) _{k=1}^{n}\right\rangle .
\end{eqnarray*}%
From these relations, we deduce that $\varphi _{T}$ verifies the conditions $%
(2)$ and $(3)$ from Definition 2.1.

Let $g\in G$. From 
\begin{eqnarray*}
&&\left\langle \left[ \varphi _{T}\left( \alpha \left( gg_{i}\right)
^{-1}gg_{j}\right) \right] _{i,j=1}^{n}\left( \xi _{k}\right)
_{k=1}^{n},\left( \xi _{k}\right) _{k=1}^{n}\right\rangle \\
&=&\tsum\limits_{i,j=1}^{n}\left\langle V_{\varphi }^{\ast }T\pi _{\varphi
}\left( \alpha \left( gg_{i}\right) ^{-1}\right) \pi _{\varphi }\left(
gg_{j}\right) V_{\varphi }\xi _{j},\xi _{i}\right\rangle \\
&=&\tsum\limits_{i,j=1}^{n}\left\langle T\pi _{\varphi }\left( gg_{j}\right)
V_{\varphi }\xi _{j},\mathcal{J}_{\varphi }\pi _{\varphi }\left( \alpha
\left( gg_{i}\right) \right) \mathcal{J}_{\varphi }V_{\varphi }\xi
_{i}\right\rangle \\
&=&\tsum\limits_{i,j=1}^{n}\left\langle T\pi _{\varphi }\left( g\right) \pi
_{\varphi }\left( g_{j}\right) V_{\varphi }\xi _{j},\pi _{\varphi }\left(
gg_{i}\right) V_{\varphi }\xi _{i}\right\rangle \\
&=&\left\langle \pi _{\varphi }\left( g\right) ^{\ast }\pi _{\varphi }\left(
g\right) \tsum\limits_{j=1}^{n}\left\vert T\right\vert \pi _{\varphi }\left(
g_{j}\right) V_{\varphi }\xi _{j},\tsum\limits_{i=1}^{n}\left\vert
T\right\vert \pi _{\varphi }\left( g_{i}\right) V_{\varphi }\xi
_{i}\right\rangle \\
&\leq &\left\Vert \pi _{\varphi }\left( g\right) \right\Vert
^{2}\left\langle T\tsum\limits_{j=1}^{n}\pi _{\varphi }\left( g_{j}\right)
V_{\varphi }\xi _{j},\tsum\limits_{i=1}^{n}\pi _{\varphi }\left(
g_{i}\right) V_{\varphi }\xi _{i}\right\rangle \\
&=&\left\Vert \pi _{\varphi }\left( g\right) \right\Vert ^{2}\left\langle 
\left[ \varphi _{T}\left( \alpha \left( g_{i}^{-1}\right) g_{j}\right) %
\right] _{i,j=1}^{n}\left( \xi _{k}\right) _{k=1}^{n},\left( \xi _{k}\right)
_{k=1}^{n}\right\rangle
\end{eqnarray*}%
for all $g_{1},...,g_{n}\in G$ and $\xi _{1},...,\xi _{n}\in \mathcal{H}$,
we deduce that $\varphi _{T}$ verifies the condition $(4)$ from Definition
2.1.
\end{proof}

Let $\varphi ,\psi $ be two $\alpha $-completely positive maps. We say that$%
\ \psi \leq \varphi \ \ $if $\varphi -\psi $ is $\alpha $-completely
positive map, and $\psi $ is \textit{uniformly dominated} by $\varphi $,
denoted by $\psi \leq _{u}\varphi $, if there is $\lambda >0$ such that $%
\psi \leq \lambda \varphi $. The $\alpha $-completely positive maps $\varphi
,\psi $ are \textit{uniformly equivalent, }$\psi \equiv _{u}\varphi $, if $%
\psi \leq _{u}\varphi $ and $\varphi \leq _{u}\psi $.

\begin{proposition}
Let $\varphi ,\psi $ be two $\alpha $-completely positive maps from $G$ to $%
L(\mathcal{H})$. If $\psi \leq _{u}\varphi $,$\ $then there is $T\in $ $\pi
_{\varphi }\left( G\right) ^{\prime }\subseteq L\left( \mathcal{H}_{\varphi
}\right) ,\ T\geq 0\ $and $T\mathcal{J}_{\varphi }=\mathcal{J}_{\varphi }T$
such that $\psi =\varphi _{T}$. Moreover, $T\ $is unique.
\end{proposition}

\begin{proof}
Let $\left( \pi _{\psi },\left( \mathcal{H}_{\psi },\mathcal{J}_{\psi
}\right) ,V_{\phi }\right) $ be the minimal Stinespring construction
associated to $\psi $. From%
\begin{eqnarray*}
&&\left\langle \tsum_{i=1}^{n}\pi _{\psi }\left( g_{i}\right) V_{\psi }\xi
_{i},\tsum_{i=1}^{n}\pi _{\psi }\left( g_{i}\right) V_{\psi }\xi
_{i}\right\rangle \\
&=&\tsum_{i,j=1}^{n}\left\langle V_{\psi }^{\ast }\pi _{\psi }\left(
g_{j}\right) ^{\ast }\pi _{\psi }\left( g_{i}\right) V_{\psi }\xi _{i},\xi
_{j}\right\rangle \\
&=&\tsum_{i,j=1}^{n}\left\langle V_{\psi }^{\ast }\pi _{\psi }\left( \alpha
\left( g_{j}^{-1}\right) g_{i}\right) V_{\psi }\xi _{i},\xi _{j}\right\rangle
\\
&=&\left\langle \left[ \psi \left( \alpha \left( g_{j}^{-1}\right)
g_{i}\right) \right] _{i,j=1}^{n}\left( \xi _{k}\right) _{k=1}^{n},\left(
\xi _{k}\right) _{k=1}^{n}\right\rangle \\
&\leq &\lambda \left\langle \left[ \varphi \left( \alpha \left(
g_{j}^{-1}\right) g_{i}\right) \right] _{i,j=1}^{n}\left( \xi _{k}\right)
_{k=1}^{n},\left( \xi _{k}\right) _{k=1}^{n}\right\rangle \\
&=&\lambda \left\langle \tsum_{i=1}^{n}\pi _{\varphi }\left( g_{i}\right)
V_{\varphi }\xi _{i},\tsum_{i=1}^{n}\pi _{\varphi }\left( g_{i}\right)
V_{\varphi }\xi _{i}\right\rangle
\end{eqnarray*}%
we deduce that there is a bounded linear operator $S:\mathcal{H}_{\varphi
}\rightarrow \mathcal{H}_{\psi }$ such that $S\left( \pi _{\varphi }\left(
g\right) V_{\varphi }\xi \right) =\pi _{\psi }\left( g\right) V_{\psi }\xi $%
. Clearly, $S\pi _{\varphi }\left( g\right) =\pi _{\psi }\left( g\right) S$
for all $g\in G$, and $SV_{\varphi }=V_{\psi }$. Moreover, $S\mathcal{J}%
_{\varphi }=\mathcal{J}_{\psi }S$, since 
\begin{eqnarray*}
S\mathcal{J}_{\varphi }\left( \pi _{\varphi }\left( g\right) V_{\varphi }\xi
\right) &=&S\left( \pi _{\varphi }\left( \alpha \left( g\right) \right)
V_{\varphi }\xi \right) =\pi _{\psi }\left( \alpha \left( g\right) \right)
V_{\psi }\xi \\
&=&\mathcal{J}_{\psi }\pi _{\psi }\left( g\right) V_{\psi }\xi =\mathcal{J}%
_{\psi }S\left( \pi _{\varphi }\left( g\right) V_{\varphi }\xi \right)
\end{eqnarray*}%
for all $g\in G$ and for all $\xi \in \mathcal{H}$.

Let $T=S^{\ast }S$. Then $T\mathcal{J}_{\varphi }=\mathcal{J}_{\varphi }T$
and $T\pi _{\varphi }\left( g\right) =\pi _{\varphi }\left( g\right) T$ for
all $g\in G$. Moreover, 
\begin{eqnarray*}
\varphi _{T}(g) &=&V_{\varphi }^{\ast }T\pi _{\varphi }\left( g\right)
V_{\varphi }=V_{\varphi }^{\ast }S^{\ast }S\pi _{\varphi }\left( g\right)
V_{\varphi }=V_{\varphi }^{\ast }S^{\ast }\pi _{\psi }\left( g\right)
V_{\varphi } \\
&=&V_{\varphi }^{\ast }S^{\ast }\mathcal{J}_{\psi }\pi _{\psi }\left(
g^{-1}\right) ^{\ast }\mathcal{J}_{\psi }V_{\psi }=V_{\varphi }^{\ast }%
\mathcal{J}_{\varphi }\pi _{\varphi }\left( g^{-1}\right) ^{\ast }S^{\ast
}V_{\psi } \\
&=&V_{\varphi }^{\ast }\pi _{\varphi }\left( g^{-1}\right) ^{\ast }S^{\ast
}V_{\psi }=V_{\psi }^{\ast }\pi _{\psi }\left( g^{-1}\right) ^{\ast }V_{\psi
}=V_{\psi }^{\ast }\pi _{\psi }\left( g\right) V_{\psi }=\psi \left( g\right)
\end{eqnarray*}%
for all $g\in G$.

Suppose that there is another $T_{1}\in \pi _{\varphi }\left( G\right)
^{\prime }\subseteq L\left( \mathcal{H}_{\varphi }\right) $, $\ T_{1}\geq 0\ 
$and $T_{1}\mathcal{J}_{\varphi }=\mathcal{J}_{\varphi }T_{1}$ such that $%
\psi =\varphi _{T_{1}}$. Then 
\begin{eqnarray*}
&&\left\langle \left( T-T_{1}\right) \left( \pi _{\varphi }\left( g\right)
V_{\varphi }\xi \right) ,\pi _{\varphi }\left( g^{\prime }\right) V_{\varphi
}\eta \right\rangle \\
&=&\left\langle V_{\varphi }^{\ast }\mathcal{J}_{\varphi }\pi _{\varphi
}\left( g^{\prime -1}\right) \mathcal{J}_{\varphi }\left( T-T_{1}\right)
\left( \pi _{\varphi }\left( g\right) V_{\varphi }\xi \right) ,\eta
\right\rangle \\
&=&\left\langle V_{\varphi }^{\ast }\left( T-T_{1}\right) \pi _{\varphi
}\left( g^{\prime -1}\right) \left( \pi _{\varphi }\left( \alpha \left(
g\right) \right) V_{\varphi }\xi \right) ,\eta \right\rangle \\
&=&\left\langle \varphi _{T}\left( g^{\prime -1}\alpha \left( g\right)
\right) \xi -\varphi _{T_{1}}\left( g^{\prime -1}\alpha \left( g\right)
\right) \xi ,\eta \right\rangle =0
\end{eqnarray*}%
for all $g,g^{\prime }G$, and for all $\xi ,\eta \in \mathcal{H}$, and since 
$\left[ \pi _{\varphi }\left( G\right) V_{\varphi }\mathcal{H}\right] =%
\mathcal{H}_{\varphi }$, we have $T=T_{1}$.
\end{proof}

From the proof of Proposition 3.1, we obtaine the following corollary.

\begin{corollary}
If $\varphi ,\psi $ are two $\alpha $-completely positive maps from $G$ to $%
L(\mathcal{H})$ and $\psi \leq \varphi $,$\ $then there is a unique positive
operator $T$ in $\pi _{\varphi }\left( G\right) ^{\prime }\subseteq L\left( 
\mathcal{H}_{\varphi }\right) $ such that $\ T\leq $id$_{\mathcal{H}%
_{\varphi }}$,$\ T\mathcal{J}_{\varphi }=\mathcal{J}_{\varphi }T$ and $\psi
=\varphi _{T}$.
\end{corollary}

Let $\varphi ,\psi $ be two $\alpha $-completely positive maps from $G$ to $%
L(\mathcal{H})$ such that $\psi \leq _{u}\varphi $. A positive operator $%
T\in $ $\pi _{\varphi }\left( G\right) ^{\prime }\subseteq L\left( \mathcal{H%
}_{\varphi }\right) $ with $T\mathcal{J}_{\varphi }=\mathcal{J}_{\varphi }T$
and such that $\psi =\varphi _{T}$, denoted by $\Delta _{\varphi }\left(
\psi \right) $, is called \textit{the} \textit{Radon-Nikodym derivative of }$%
\psi $\textit{\ with respect to }$\varphi .$\textit{\ }

\begin{remark}
If $\psi \leq _{u}\varphi $, then the minimal Stinespring construction
associated to $\psi $ can be recovered by the minimal Stinespring
construction associated to $\varphi .$

Let $P_{\ker \Delta _{\varphi }\left( \psi \right) }$ and $P_{\mathcal{H}%
_{\varphi }\ominus \ker \Delta _{\varphi }\left( \psi \right) }$ be the
orthogonal projections on $\ker \Delta _{\varphi }\left( \psi \right) ,$
respectively $\mathcal{H}_{\varphi }\ominus \ker \Delta _{\varphi }\left(
\psi \right) $. Since $\Delta _{\varphi }\left( \psi \right) \in \pi
_{\varphi }\left( G\right) ^{\prime }\subseteq L\left( \mathcal{H}_{\varphi
}\right) $ and $\Delta _{\varphi }\left( \psi \right) \mathcal{J}_{\varphi }=%
\mathcal{J}_{\varphi }\Delta _{\varphi }\left( \psi \right) ,$ $P_{\ker
\Delta _{\varphi }\left( \psi \right) },$ $P_{\mathcal{H}_{\varphi }\ominus
\ker \Delta _{\varphi }\left( \psi \right) }\in \pi _{\varphi }\left(
G\right) ^{\prime }\subseteq L\left( \mathcal{H}_{\varphi }\right) $,$%
P_{\ker \Delta _{\varphi }\left( \psi \right) }\mathcal{J}_{\varphi }=%
\mathcal{J}_{\varphi }P_{\ker \Delta _{\varphi }\left( \psi \right) }$ and $%
P_{\mathcal{H}_{\varphi }\ominus \ker \Delta _{\varphi }\left( \psi \right) }%
\mathcal{J}_{\varphi }=\mathcal{J}_{\varphi }P_{\mathcal{H}_{\varphi
}\ominus \ker \Delta _{\varphi }\left( \psi \right) }$. Then $\left( 
\mathcal{H}_{\varphi }\ominus \ker \Delta _{\varphi }\left( \psi \right) ,%
\mathcal{J}_{\varphi }|_{\mathcal{H}_{\varphi }\ominus \ker \Delta _{\varphi
}\left( \psi \right) }\right) $ is a Krein space and it is easy to check
that 
\begin{equation*}
\left( \pi _{\varphi }|_{\mathcal{H}_{\varphi }\ominus \ker \Delta _{\varphi
}\left( \psi \right) }\left( \mathcal{H}_{\varphi }\ominus \ker \Delta
_{\varphi }\left( \psi \right) ,\mathcal{J}_{\varphi }|_{\mathcal{H}%
_{\varphi }\ominus \ker \Delta _{\varphi }\left( \psi \right) }\right) ,P_{%
\mathcal{H}_{\varphi }\ominus \ker \Delta _{\varphi }\left( \psi \right)
}\Delta _{\varphi }\left( \psi \right) ^{\frac{1}{2}}V_{\varphi }\right)
\end{equation*}
is unitarily equivalent to the minimal Stinespring construction associated
to $\psi .$
\end{remark}

\begin{proposition}
Let $\varphi ,\psi \in \alpha -CP(G,\mathcal{H})$. If $\varphi \equiv
_{u}\psi $, then the Stinespring construction associated to $\varphi $ is
unitarily equivalent to the Stinespring construction associated to $\varphi
. $
\end{proposition}

\begin{proof}
If $\varphi \equiv _{u}\psi $, then $\psi \leq _{u}\varphi $ and $\varphi
\leq _{u}\psi $, and by Proposition 3.2, there are two bounded linear
operators $S_{1}:$ $\mathcal{H}_{\varphi }\rightarrow \mathcal{H}_{\psi }$
such that $S_{1}\left( \pi _{\varphi }\left( g\right) V_{\varphi }\xi
\right) =\pi _{\psi }\left( g\right) V_{\psi }\xi $ and $S_{2}:\mathcal{H}%
_{\psi }\rightarrow \mathcal{H}_{\varphi }$ such that $S_{2}\left( \pi
_{\psi }\left( g\right) V_{\psi }\xi \right) =\pi _{\varphi }\left( g\right)
V_{\varphi }\xi .$ From $S_{2}S_{1}\left( \pi _{\varphi }\left( g\right)
V_{\varphi }\xi \right) =\pi _{\varphi }\left( g\right) V_{\varphi }\xi ,$ $%
S_{1}S_{2}\left( \pi _{\psi }\left( g\right) V_{\psi }\xi \right) =\pi
_{\psi }\left( g\right) V_{\psi }\xi $, and taking into account that $\left[
\pi _{\varphi }\left( G\right) V_{\varphi }\mathcal{H}\right] =\mathcal{H}%
_{\varphi }$ and $\left[ \pi _{\psi }\left( G\right) V_{\psi }\mathcal{H}%
\right] =\mathcal{H}_{\psi }$, we deduce that $S_{1}$ is invertible. Then $%
\Delta _{\varphi }\left( \psi \right) =S_{1}^{\ast }S_{1}$ is invertible,
and so there is a unitary operator $U:\mathcal{H}_{\varphi }\rightarrow 
\mathcal{H}_{\psi }$ such that $S_{1}=U\Delta _{\varphi }\left( \psi \right)
^{\frac{1}{2}}$. It is easy to check that $U\mathcal{J}_{\varphi }=\mathcal{J%
}_{\psi }U,UV_{\varphi }=V_{\psi }\ $and $U\pi _{\varphi }(g)=\pi _{\psi
}\left( g\right) U$ for all $g\in G.$
\end{proof}

\begin{theorem}
Let $\varphi $ be an $\alpha $-completely positive map from $G$ to $L(%
\mathcal{H})$. Then the map $\psi \mapsto \Delta _{\varphi }\left( \psi
\right) $ is an affine bijective map from $\{\psi \in \alpha -CP(G,\mathcal{H%
});\psi \leq _{u}\varphi \}\ $onto $\{T\in \pi _{\varphi }\left( G\right)
^{\prime }\subseteq L(\mathcal{H}_{\varphi });T\mathcal{J}_{\varphi }=%
\mathcal{J}_{\varphi }T,T\geq 0\}$ which preserves the pre-order relation.
\end{theorem}

\begin{proof}
By Propositions 3.1 and 3.2, the map $\psi \mapsto \Delta _{\varphi }\left(
\psi \right) $ is well defined and bijective, its inverse is given by $%
T\mapsto \varphi _{T}$. Let $t\in \lbrack 0,1],\psi _{1}\leq _{u}\varphi $
and $\psi _{2}\leq _{u}\varphi $. Then $t\psi _{1}+(1-t)\psi _{2}\leq
_{u}\varphi $ and so 
\begin{eqnarray*}
\varphi _{\Delta _{\varphi }\left( t\psi _{1}+(1-t)\psi _{2}\right) }\left(
g\right) &=&t\psi _{1}\left( g\right) +(1-t)\psi _{2}\left( g\right)
=t\varphi _{\Delta _{\varphi }\left( \psi _{1}\right) }\left( g\right)
+(1-t)\varphi _{\Delta _{\varphi }\left( \psi _{2}\right) }\left( g\right) \\
&=&V_{\varphi }^{\ast }\left( t\Delta _{\varphi }\left( \psi _{1}\right)
+(1-t)\Delta _{\varphi }\left( \psi _{2}\right) \right) \pi _{\varphi
}\left( g\right) V_{\varphi }
\end{eqnarray*}%
for all $g\in G$, whence we deduce that $\Delta _{\varphi }\left( t\psi
_{1}+(1-t)\psi _{2}\right) =t\Delta _{\varphi }\left( \psi _{1}\right)
+(1-t)\Delta _{\varphi }\left( \psi _{2}\right) .$ Therefore, the map $\psi
\mapsto \Delta _{\varphi }\left( \psi \right) $ is affine.

Let $\psi _{1}\leq _{u}\psi _{2}\leq _{u}\varphi $. Then there is $\lambda
\geq 0$ such that $\lambda \psi _{2}-\psi _{1}$ is $\alpha $-completely
positive. From 
\begin{eqnarray*}
0 &\leq &\left\langle \left[ \left( \lambda \psi _{2}-\psi _{1}\right)
\left( \alpha \left( g_{i}\right) ^{-1}g_{j}\right) \right]
_{i,j=1}^{n}\left( \xi _{i}\right) _{i=1}^{n},\left( \xi _{i}\right)
_{i=1}^{n}\right\rangle \\
&=&\lambda \tsum\limits_{i,j=1}^{n}\left\langle \varphi _{\Delta _{\varphi
}\left( \psi _{2}\right) }\left( \alpha \left( g_{i}\right)
^{-1}g_{j}\right) \xi _{j},\xi _{i}\right\rangle
-\tsum\limits_{i,j=1}^{n}\left\langle \varphi _{\Delta _{\varphi }\left(
\psi _{1}\right) }\left( \alpha \left( g_{i}\right) ^{-1}g_{j}\right) \xi
_{j},\xi _{i}\right\rangle \\
&=&\tsum\limits_{i,j=1}^{n}\left\langle V_{\varphi }^{\ast }\left( \lambda
\Delta _{\varphi }\left( \psi _{2}\right) -\Delta _{\varphi }\left( \psi
_{1}\right) \right) \pi _{\varphi }\left( \alpha \left( g_{i}\right)
^{-1}g_{j}\right) V_{\varphi }\xi _{j},\xi _{i}\right\rangle \\
&=&\tsum\limits_{i,j=1}^{n}\left\langle \left( \lambda \Delta _{\varphi
}\left( \psi _{2}\right) -\Delta _{\varphi }\left( \psi _{1}\right) \right)
\pi _{\varphi }\left( g_{j}\right) V_{\varphi }\xi _{j},\pi _{\varphi
}\left( \alpha \left( g_{i}\right) ^{-1}\right) ^{\ast }V_{\varphi }\xi
_{i}\right\rangle \\
&=&\tsum\limits_{i,j=1}^{n}\left\langle \left( \lambda \Delta _{\varphi
}\left( \psi _{2}\right) -\Delta _{\varphi }\left( \psi _{1}\right) \right)
\pi _{\varphi }\left( g_{j}\right) V_{\varphi }\xi _{j},\mathcal{J}_{\varphi
}\pi _{\varphi }\left( \alpha \left( g_{i}\right) \right) \mathcal{J}%
_{\varphi }V_{\varphi }\xi _{i}\right\rangle \\
&=&\tsum\limits_{i,j=1}^{n}\left\langle \left( \lambda \Delta _{\varphi
}\left( \psi _{2}\right) -\Delta _{\varphi }\left( \psi _{1}\right) \right)
\pi _{\varphi }\left( g_{j}\right) V_{\varphi }\xi _{j},\mathcal{J}_{\varphi
}\pi _{\varphi }\left( \alpha \left( g_{i}\right) \right) V_{\varphi }\xi
_{i}\right\rangle \\
&=&\left\langle \left( \lambda \Delta _{\varphi }\left( \psi _{2}\right)
-\Delta _{\varphi }\left( \psi _{1}\right) \right) \tsum\limits_{j=1}^{n}\pi
_{\varphi }\left( g_{j}\right) V_{\varphi }\xi
_{j},\tsum\limits_{i=1}^{n}\pi _{\varphi }\left( g_{i}\right) V_{\varphi
}\xi _{i}\right\rangle
\end{eqnarray*}%
for all $g_{1},...,g_{n}\in G$, for all $\xi _{1},...,\xi _{n}\in \mathcal{H}%
,$ and taking into account that $\left[ \pi _{\varphi }\left( G\right)
V_{\varphi }\mathcal{H}\right] =\mathcal{H}_{\varphi }$, we conclude that $%
\lambda \Delta _{\varphi }\left( \psi _{2}\right) -\Delta _{\varphi }\left(
\psi _{1}\right) \geq 0$. Therefore the map $\psi \mapsto \Delta _{\varphi
}\left( \psi \right) $ preserves the pre-order relation.
\end{proof}

\begin{corollary}
The map $\psi \mapsto \Delta _{\varphi }\left( \psi \right) $ is an affine
bijective map from $\{\psi \in \alpha -CP(G,\mathcal{H});\psi \leq \varphi
\}\ $onto $\{T\in \pi _{\varphi }\left( G\right) ^{\prime }\subseteq L(%
\mathcal{H}_{\varphi });T\mathcal{J}_{\varphi }=\mathcal{J}_{\varphi
}T,0\leq T\leq $id$_{\mathcal{H}_{\varphi }}\}$ which preserves the order
relation.
\end{corollary}

Let $G$ be a discrete group. If $\pi $ is a bounded $\mathcal{J}$-unitary
representation of $G$ on $\left( \mathcal{H},\mathcal{J}\right) $, then the
map $\widetilde{\pi }:\mathcal{F}\left( G\right) \rightarrow L(\mathcal{H})$
given by%
\begin{equation*}
\widetilde{\pi }\left( \tsum\limits_{i=1}^{n}\lambda _{i}\delta
_{g_{i}}\right) =\tsum\limits_{i=1}^{n}\lambda _{i}\pi \left( g_{i}\right)
\end{equation*}%
extends to a bounded $\mathcal{J}$- representation of $C^{\ast }(G)$.
Moreover, the map $\pi \mapsto \widetilde{\pi }$ is a bijective
correspondence between the collection of bounded unitary representations on
Krein spaces and the collection of bounded representations of $C^{\ast }(G)$
on Krein spaces.

Let $G$ be a discrete group and $\varphi \in \alpha -CP(G,\mathcal{H})$.
Then $\alpha $ extends to a linear hermitian involution $\widetilde{\alpha }$
on $C^{\ast }(G),$ $\widetilde{\alpha }\left( f\right) =f\circ \alpha $ for
all $f\in \mathcal{F}\left( G\right) $. If $\varphi $ is bounded, then the
map $\Phi :\mathcal{F}\left( G\right) \rightarrow L(\mathcal{H})$ given by $%
\Phi \left( \tsum\limits_{k=1}^{n}\lambda _{k}\delta _{g_{k}}\right)
=\tsum\limits_{k=1}^{n}\lambda _{k}\varphi \left( g_{k}\right) $ extends to
a linear hermitian bounded $\widetilde{\alpha }$-completely positive $%
\widetilde{\varphi }:C^{\ast }(G)\rightarrow L(\mathcal{H})$ (see \cite[%
Theorem 2.5]{H}). We denote by $\alpha -bCP(G,\mathcal{H})$ the collection
of all bounded $\alpha $-completely positive maps from $G$ to $L(\mathcal{H}%
).$

\begin{remark}
Let $\varphi \in \alpha -bCP(G,\mathcal{H})$. If $\left( \pi _{\varphi
},\left( \mathcal{H}_{\varphi },\mathcal{J}_{\varphi }\right) ,V_{\varphi
}\right) $ is the minimal Stinespring construction associated to $\varphi $,
then it is easy to check that $\left( \widetilde{\pi _{\varphi }},\left( 
\mathcal{H}_{\varphi },\mathcal{J}_{\varphi }\right) ,V_{\varphi }\right) $
is unitarily equivalent to the minimal Stinespring construction associated
to $\widetilde{\varphi }\ $(\cite[Theorems 4.4 and 4.6 ]{HHJ})$.$
\end{remark}

\begin{theorem}
Let $G$ be a discrete group. Then the map $\varphi \mapsto \widetilde{%
\varphi }$ is an affine bijective map from $\alpha -bCP(G,\mathcal{H})$ to $%
\alpha -bCP(C^{\ast }(G),\mathcal{H})$ which preserves the order (pre-order)
relation. Moreover, if $\psi \leq _{u}\varphi $ then $\Delta _{\varphi
}\left( \psi \right) =\Delta _{\widetilde{\varphi }}\left( \widetilde{\psi }%
\right) .$
\end{theorem}

\begin{proof}
It is clear that the map $\varphi \mapsto \widetilde{\varphi }$ from $\alpha
-bCP(G,\mathcal{H})$ to $\alpha -bCP(C^{\ast }(G),\mathcal{H})$ is well
defined and injective. Let $\phi \in \alpha -bCP(C^{\ast }(G),\mathcal{H})$.
Then the map $\varphi :G\rightarrow L(\mathcal{H})$ given $\varphi \left(
g\right) =\phi \left( \delta _{g}\right) $ is a bounded $\alpha $-completely
positive. Moreover, $\widetilde{\varphi }=\phi $, and so the map $\varphi
\mapsto \widetilde{\varphi }$ is surjective.

Clearly, $\widetilde{\varphi _{1}+\varphi _{2}}=\widetilde{\varphi _{1}}+%
\widetilde{\varphi _{2}}$ and $\widetilde{\lambda \varphi }=\lambda 
\widetilde{\varphi }$ for all $\varphi _{1},\varphi _{2},\varphi \in \alpha
-bCP(G,\mathcal{H})$ and for all positive numbers $\lambda $. Let $\varphi
,\psi \in \alpha -bCP(G,\mathcal{H})$ with $\psi \leq \varphi $ and $\left(
\pi _{\varphi },\left( \mathcal{H}_{\varphi },\mathcal{J}_{\varphi }\right)
,V_{\varphi }\right) $ the minimal Stinespring construction associated to $%
\varphi $. Then $\widetilde{\psi }\left( f\right) =V_{\varphi }^{\ast
}\Delta _{\varphi }\left( \psi \right) \widetilde{\pi _{\varphi }}\left(
f\right) V_{\varphi }$ for all $f\in C^{\ast }(G)$ and since $\left( 
\widetilde{\pi _{\varphi }},\left( \mathcal{H}_{\varphi },\mathcal{J}%
_{\varphi }\right) ,V_{\varphi }\right) $ is unitarily equivalent to the
minimal Stinespring construction associated to $\widetilde{\varphi }$,$\ $
and since $\Delta _{\varphi }\left( \psi \right) \in \widetilde{\pi
_{\varphi }}\left( G\right) ^{\prime }\subseteq L(\mathcal{H}_{\varphi
}),\Delta _{\varphi }\left( \psi \right) \mathcal{J}_{\varphi }=\mathcal{J}%
_{\varphi }\Delta _{\varphi }\left( \psi \right) $ and $0\leq \Delta
_{\varphi }\left( \psi \right) \leq $id$_{\mathcal{H}_{\varphi }},\widetilde{%
\psi }\leq \widetilde{\varphi }$ and $\Delta _{\varphi }\left( \psi \right)
=\Delta _{\widetilde{\varphi }}\left( \widetilde{\psi }\right) .$
\end{proof}

\end{document}